\documentclass[english]{article}
\usepackage[T1]{fontenc}
\usepackage[latin9]{inputenc}
\usepackage{color}
\usepackage{babel}

\usepackage{amsthm}
\usepackage{amsmath}
\usepackage{amssymb}
\usepackage{esint}
\usepackage{hyperref}
\usepackage{fullpage}
\usepackage{enumerate}

\makeatletter

\numberwithin{equation}{section}

\theoremstyle{plain}
\newtheorem{thm}{Theorem}[section]
\theoremstyle{remark}
\newtheorem{rem}[thm]{Remark}
\theoremstyle{remark}

\theoremstyle{plain}
\newtheorem{lem}[thm]{Lemma}
\theoremstyle{plain}
\newtheorem{prop}[thm]{Proposition}
\theoremstyle{plain}
\newtheorem{cor}[thm]{Corollary}

\makeatother

\begin{document}

\title{A note on Talagrand's variance bound in terms of influences}
%\title{A generalization of Talagrand's variance bound in terms of influences}

\author{Demeter Kiss \thanks{CWI; research supported by NWO; e-mail: D.Kiss@cwi.nl}}
\maketitle
\begin{abstract}
Let $X_1,\ldots, X_n$ be independent Bernoulli random variables and $f$ a function on $\{0,1\}^n.$ 
In the well-known paper \cite{Talagrand1994} Talagrand gave an upper bound for the variance of $f$ in terms of 
the individual influences of the $X_i\mbox{'s}.$ This bound turned out to be very useful, for instance
in percolation theory and related fields.

In many situations a similar bound was needed for random variables taking more than two values. 
Generalizations of this type have indeed been obtained in the literature (see e.g. \cite{Cordero-Erausquin2011}),
but the proofs are quite different from that in \cite{Talagrand1994}. This might raise the impression that
Talagrand's original method is not sufficiently robust to obtain such generalizations.

However, our paper gives an almost self-contained proof of the above mentioned generalization,
by modifying step-by-step Talagrand's original proof.

%  Consider a random variable of the form $f\left(X_{1},\ldots,X_{n}\right),$
%  where $f$ is a deterministic function, and where $X_{1},\ldots,X_{n}$
%  are i.i.d random variables. For the case where $X_{1}$ has a Bernoulli
%  distribution, Talagrand (in \cite{Talagrand1994}) gave an upper bound
%  for the variance of $f$ in terms of the individual influences of
%  the variables $X_{i}$ for $i=1,\ldots,n.$ Our result is a generalization
%  to the case where $X_{1}$ takes finitely many values.
%
%  A more general result can be found in \cite{Cordero-Erausquin2011}, however the 
%  approach there is quite different from the original proof of Talagrand, which we extend in this paper.

%  A result closely related to ours was stated as Theorem 3.3 in \cite{Hatami2009}. However, we show that the result is incorrect (too general).
%  We also show that a less general, modified version follows from our main result. The modified version does not follow from the arguments of \cite{Hatami2009},
%  since a key ingredient is missing in that paper.
\end{abstract}


\section{Introduction and statement of results}

\subsection{Statement of the main results}

Let $\left(\Omega,\mathcal{F},\mu\right)$ be an arbitrary probability
space. We denote its $n$-fold product by itself by $\left(\Omega^{n},\mathcal{F}^{n},\mu^{n}\right).$
Let $f:\Omega^{n}\rightarrow\mathbb{C}$ be a function with finite
second moment, that is $\int_{\Omega^{n}}|f|^{2}d\mu^{n}<\infty.$
The influence of the $i$th variable on the function $f$ is defined as
\[
  \Delta_{i}f\left(x_{1},\ldots,x_{n}\right)=f\left(x_{1},\ldots,x_{n}\right)-\int_{\Omega}f\left(x_{1},\ldots,x_{i-1},\xi,x_{i+1},\ldots,x_{n}\right)\mu\left(d\xi\right)
\]
for $x=\left(x_{1},\ldots,x_{n}\right)\in\Omega^{n}$ and $i=1,\ldots,n.$
We will use the notation $\left\Vert f\right\Vert _{q}$ for the $L^{q}$
norm $q\in[1,\infty)$ of $f,$ that is $\left\Vert f\right\Vert _{q}=\sqrt[q]{\int_{\Omega^{n}}|f|^{q}d\mu^{n}}.$

Using Jensen's inequality, Efron and Stein gave the following upper bound on the variance of $f$ (see \cite{Efron1981}):
\begin{equation}
  Var\left(f\right)\leq\sum_{i=1}^{n}\left\Vert \Delta_{i}f\right\Vert _{2}^{2}.\label{eq:Efron-Stein}
\end{equation}

In some cases \eqref{eq:Efron-Stein} has been improved.
We write $\mathcal{P}\left(S\right)$ for the power set of a set $S.$
For the case when $\Omega$ has two elements, say $0$ and $1,$ and
$\mu\left(\left\{ 1\right\} \right)=1-\mu\left(\left\{ 0\right\} \right)=p$,
Talagrand showed the following result:
\begin{thm}[Theorem 1.5 of \cite{Talagrand1994}] \label{thm:thm 1.5 of Talagrand}
  There exists a universal constant $K$ such
  that for every $p\in\left(0,1\right),$ $n\in\mathbb{N}$ and for
  every real valued function $f$ on %$\left(\Omega^{n},\mathcal{F}^{n},\mu^{n}\right)=$
$\left(\left\{ 0,1\right\} ^{n},\mathcal{P}\left(\left\{ 0,1\right\} ^{n}\right),\mu_{p}\right),$
  \begin{equation}
    Var\left(f\right)\leq K\log\left(\frac{2}{p(1-p)}\right)\sum_{i=1}^{n}\frac{\left\Vert \Delta_{i}f\right\Vert _{2}^{2}}{\log\left(e\left\Vert \Delta_{i}f\right\Vert _{2}/\left\Vert \Delta_{i}f\right\Vert _{1}\right)},\label{eq:Talineq}
  \end{equation}
  where $\mu_{p}$ is the product measure on $\{0,1\}^n$ with parameter $p.$
%=p^{\sum_{i=1}^{n}x_{i}}\left(1-p\right)^{n-\sum_{i=1}^{n}x_{i}}$ for $x\in\left\{ 0,1\right\} ^{n}.$
\end{thm}

\begin{rem}
  An alternative proof of Theorem \ref{thm:thm 1.5 of Talagrand} for
  the case $p=1/2$ can be found in \cite{Benjamini2003}.
\end{rem}

Inequality \eqref{eq:Talineq} gives a bound on $Var\left(f\right)$
in terms of the influences. It is useful when the function $f$ is complicated, but its influences
are tractable. Such situations occur for example in percolation
theory (see for example \cite{Benjamini2003,Bollobas2008,Berg2008}).
Further consequences of \eqref{eq:Talineq} include for example
the widely used KKL lower bound for influences \cite{Kahn1988}
and various so called sharp-threshold results e.g. \cite{Friedgut1996}.

In some cases, a generalization of Theorem \ref{thm:thm 1.5 of Talagrand} to the case $\{0,1,\ldots, k\}^n$ with $k>1$ is useful,
for example in \cite{Bollobas2010,Devroye2008}. However, up to our knowledge,
no such generalization has been explicitly stated in the literature. The main goal of our paper is to present
and prove an explicit generalization, Theorem \ref{thm:Thm1.5gen} below. We have used this theorem and referred to it in \cite{Kiss}.
\begin{thm}
  \label{thm:Thm1.5gen}There is a universal constant
  $K>0$ such that for each finite set $\Omega$ each measure $\mu$
  on $\Omega$ with $p_{min}=\min_{j\in\Omega}\mu\left(\left\{ j\right\} \right)>0,$
  and for all complex valued functions $f$ on $\left(\Omega^{n},\mathcal{P}\left(\Omega^{n}\right),\mu^{n}\right),$
  \begin{equation}
    Var\left(f\right)\leq K\log\left(1/p_{min}\right)\sum_{i=1}^{n}\frac{\left\Vert \Delta_{i}f\right\Vert _{2}^{2}}{\log\left(e\left\Vert \Delta_{i}f\right\Vert _{2}/\left\Vert \Delta_{i}f\right\Vert _{1}\right)}.\label{eq:gen talineq}
  \end{equation}
\end{thm}
\begin{rem}\label{rem: Thm1.5gen}
  Inequality \eqref{eq:gen talineq} is sharp up to a universal constant factor, which can easily be seen 
  by taking the function $f(x)=1$ if $x_i=\omega$ for all $i=1,\ldots,n$ where $\omega$ is some element of $\Omega$ is such that 
  $\mu\left(\left\{ \omega \right\} \right)=p_{min},$ and $f(x)=0$ otherwise.
\end{rem}

Herein, we follow the line of argument of Talagrand  \cite{Talagrand1994} and modify his symmetrization procedure to deduce the result above.
Given the paper of Talagrand \cite{Talagrand1994}, the proof is self contained apart from Lemma 1 of \cite{Devroye2008}.

Cordero-Erausquin and Ledoux \cite{Cordero-Erausquin2011}  in a recent preprint further generalized Theorem \ref{thm:Thm1.5gen},
however their approach is very different from the original proof of Talagrand. (One can deduce a result, equivalent up to a universal constant
to Theorem \ref{thm:Thm1.5gen}, from  Theorem 1 of \cite{Cordero-Erausquin2011}, by combining it with Theorem A.1 of \cite{Diaconis1996}. 
This results in a slightly more complicated proof.)

We finish this section by noting that the special case of Theorem \ref{thm:Thm1.5gen},
where $\mu^{n}$ is the uniform measure on $\Omega^{n},$ has been proved in \cite{Devroye2008}.

\subsection{Background and further motivation for Theorem 1.3}

%Inequality \eqref{eq:Talineq} gives a bound on $Var\left(f\right)$
%in terms of the influences. As mentioned before, related inequalities include for example
%the widely used KKL lower bound for influences \cite{Kahn1988}
%and various so called sharp-threshold results e.g. \cite{Friedgut1996}. These
%are useful when the function $f$ is complicated, but its influences
%are tractable. Such situations occur for example in percolation
%theory (see for example \cite{Benjamini2003,Bollobas2008,Bollobas2010,Berg2008}).
%Theorem \ref{thm:Thm1.5gen} is explicitly used in \cite{Kiss}.
%In fact the above mentioned KKL bound is a consequence
%of \eqref{eq:Talineq} (see Corollary 1.4 in \cite{Talagrand1994}).
%(We write 'in some sense' because the paper \cite{Kahn1988} appeared significantly earlier.)
%This also holds for certain sharp-threshold
%results (see Corollary 1.3 in \cite{Talagrand1994}).

%Inequality \eqref{eq:gen talineq} is the most literal extension of
%\eqref{eq:Talineq} to the case where $\Omega$ has cardinality $k,$
%$k\geq2.$ It is explicitly used in \cite{Kiss}.

%\medskip

Falik and Samarodnitsky \cite{Falik2007} used logarithmic Sobolev
inequalities to derive edge isoperimetric inequalities. Rossignol
used this method to derive sharp threshold results \cite{Rossignol2005,Rossignol2008}.
Furthermore,  Bena\"{i}m and Rossignol \cite{Benaim2008} extended the results of \cite{Benjamini2003}
(where Talagrand's Theorem \ref{thm:thm 1.5 of Talagrand} above is
applied to first-passage percolation), again
with the use of logarithmic Sobolev inequalities. These similar applications
suggest a deeper connection between logarithmic Sobolev inequalities
and \eqref{eq:Talineq}. Indeed, Bobkov and Houdr\'{e} in \cite{Bobkov1999},
proved that a version of \eqref{eq:Talineq} actually implies a logarithmic
Sobolev inequality in a continuous set-up. Moreover, Cordero-Erausquin and Ledoux in \cite{Cordero-Erausquin2011} 
showed the same implication under different assumptions. 

\medskip

Another motivation for Theorem \ref{thm:Thm1.5gen} is to point out the following mistake in the literature. We borrow the notation of \cite{Hatami2009}.
For any $x\in\Omega^n$ and $i=1,\ldots,n,$ we define 
$$s_i(x)=\left\{y\in\Omega^n\,|\,y_j=x_j \mbox{ for all } j\neq i\right\}.$$ 
For $i=1,\ldots,n,$ let $I_f(i)$ denote the probability of the event that the value of $f$ does depend on the $i$th coordinate, that is 
$$I_f(i)=\mu^n\left(\left\{x\in\Omega^n\,:\,f\mbox{ is non-constant on }s_i(x)\right\}\right).$$
The following claim, which is related to our Theorem \ref{thm:Thm1.5gen} was stated as Theorem 3.3 in \cite{Hatami2009}.
However, as we will show, this claim is incorrect.

\textit{For any probability space $\left(\Omega,\mathcal{F},\mu\right),$ and positive integer $n,$ for any square integrable function 
$f:\left(\Omega^n,\mathcal{F}^n,\mu^{n}\right)\rightarrow\mathbb{R},$ we have
\begin{equation}
  Var\left(f\right)\leq 10\sum_{i=1}^{n}\frac{\left\Vert \Delta_{i}f\right\Vert _{2}^{2}}{\log\left(1/I_f(i)\right)}. \label{eq:Hatami}
\end{equation}
}

%\medskip
%
%Below we give a counterexample to this. We also show how a modified, weaker (less general) version of Hatami's claim,
%can be obtained from our Theorem \ref{thm:Thm1.5gen}. This modified version is stated as Theorem \ref{thm:salvage}	 below.
%We emphasize that even this weaker version does not follow from the arguments in \cite{Hatami2009}. (See the Remarks following our counterexample).
%
%\medskip

One can easily see, that the following is a counterexample for this claim.
Let $k$ be an arbitrary positive integer. Take $n=2$ and consider the case where $\Omega=[0,1]$ and $\mu$ is the uniform measure.
Take the function $f$ (similar to the function in Remark \ref{rem: Thm1.5gen}) defined as $f(x_1,x_2)=1$ if $0\leq x_1,x_2\leq1/k$
and $0$ otherwise. Substituting to \eqref{eq:Hatami} and choosing $k$ large enough, we get a contradiction.

Note that we can easily salvage \eqref{eq:Hatami} under the conditions of Theorem \ref{thm:Thm1.5gen}.
If in equation \eqref{eq:Hatami} we replace the constant $10$ for $K\log(1/p_{min}),$ we get a valid statement,
since we it follows from \eqref{eq:gen talineq} by applying second moment method in the denominator.

\medskip

Most of the aforementioned applications of the inequality \eqref{eq:Talineq} are
concerned with the special case where $f=\textbf{1}_A,$ that is $f$ is the indicator function of some event $A\subseteq\Omega^n.$ 
We warn the reader about the slight inconsistency of the literature: $I_A(i)$ is called the influence of the $i$th variable on the event $A,$ 
instead of some $L^p,p\geq1$ norm of $\Delta_if=\Delta_i\textbf{1}_A,$ which is the usual influence for arbitrary functions. 
For comparison of different definitions of influence, see e.g. \cite{Keller}. 

Note that
\begin{equation}
  \Vert\Delta_i\textbf{1}_A\Vert_2^2=\Vert \Delta_i\textbf{1}_A\Vert_1\leq p_{med}\mu^n\left(A_i\right), 
\end{equation}
where $p_{med}=\max\left\{\mu(B)|B\subset\Omega, \mu(B)\leq\frac{1}{2}\right\}.$ Using this we can deduce the following generalization 
of Corollary 1.2 of \cite{Talagrand1994}.

\begin{cor}
  There is a universal constant $C>0$ such that for each finite set $\Omega$ and each measure $\mu$
  on $\Omega$ and for sets $A\subseteq\Omega^n,$
  \begin{equation}
    \sum_{i=1}^{n}I_A(i)\geq C\frac{\log\left(1/\max_iI_A(i)\right)}{p_{med}\log\left(1/p_{min}\right)}\mu^n(A)\left(1-\mu^n\left(A\right)\right). \label{eq:talineq events}
  \end{equation}
\end{cor}
Using the corollary above, one can easily deduce the sharp threshold results of \cite{Bollobas2010}.
%Notice that the inequality \eqref{eq:talineq events} is weak when $p_{min}$ is small, hence \eqref{eq:talineq events} is not suitable,
%or at least it is not comfortable to use, to deduce sharp threshold results. However, with different methods we can get a different bound: 
%With careful use of the binary encoding procedure of \cite{Bourgain1992}, from Corollary 1.2 of \cite{Talagrand1994}, one can deduce the following result.
%\begin{thm}
%  \label{thm:gen talagrand-1}There is a universal constant $C>0$ such that for each finite set $\Omega$ and each measure $\mu$
%  on $\Omega$ and for sets $A\subseteq\Omega^n,$
%  \begin{equation}
%    \sum_{i=1}^{n}\mu^n\left(A_{i}\right)\geq C\frac{\log\left(1/\varepsilon\right)}{|\Omega|\overline{p}\log\left(1/\overline{p}\right)}\mu^n(A)\left(1-\mu^n\left(A\right)\right),\label{eq:gentalineqevents}
%  \end{equation}
%  where $\overline{p}=1-\max_{j\in\Omega}\mu(\left\{j\right\})$, and $\varepsilon=\sup_{i\leq n}\mu^n\left(A_{i}\right).$
%  Moreover, if $A$ is an increasing event, we have
%  \[
%    \sum_{i=1}^{n}\mu^n\left(A_{i}\right)\geq C\frac{\log\left(1/\varepsilon\right)}{\overline{p}\log\left(1/\overline{p}\right)}\mu^n(A)\left(1-\mu^n\left(A\right)\right).
%  \]
%\end{thm}
%Due to the fact that the proof of Theorem \ref{thm:gen talagrand-1} is rather standard, we omit it. Combining Theorem \ref{thm:gen talagrand-1} with (generalized) Russo's formula (see \cite{Russo1982,Rossignol2008}), one can easily deduce sharp threshold results like \cite{Bollobas2010}.

\medskip

We finish this introduction with some remarks on the proof of Theorem
\ref{thm:Thm1.5gen}. The proof of Theorem 1.5 of \cite{Talagrand1994}
uses a hypercontractive result (Bonami-Beckner inequality, see \cite{Beckner1975})
followed by a subtle symmetrization procedure (see Step 2 and 3 of
the proof of Lemma 2.1 in \cite{Talagrand1994}). In the proof of our
more general Theorem \ref{thm:Thm1.5gen} above, we use a consequence
of the extended Bonami-Beckner inequality (for an extension of the
Bonami-Beckner inequality see Claim 3.1 in \cite{N.Alon2004}) from
\cite{Devroye2008} and then modify Talagrand's symmetrization procedure.
This generalization of Talagrand's symmetrization argument, which
covers Sections \ref{sub:basis} and \ref{sub:ext lemma 2.1} is the
main part of our proof.
%Note that the symmetrization argument
%can be replaced by a result by Wolff \cite{Wolff2007}, however its
%proof is different.

\section{\label{sec:Tal ineq pf}Proof of Theorem \ref{thm:Thm1.5gen}}

Without loss of generality, we assume that $\Omega=\mathbb{Z}_{k}$
(the integers modulo $k$) for some $k\in\mathbb{N}.$

Let $\eta$ be an arbitrary measure on $\mathbb{Z}_{k}^{n}.$
For each $\eta,$ we will write $L_{\eta}\left(\mathbb{Z}_{k}^{n}\right)$
for the (Hilbert) space of complex valued functions on $\mathbb{Z}_{k}^{n},$
with the inner product
\[
  \left\langle f,g\right\rangle _{\eta}=\int_{\mathbb{Z}_{k}^{n}}f\overline{g}d\eta\mbox{ for }f,g\in L_{\eta}\left(\mathbb{Z}_{k}^{n}\right).
\]
We will write $\left\Vert f\right\Vert _{L^{q}\left(\eta\right)}$
for the $q$-norm, $q\in\left[1,\infty\right),$ of a function $f:\mathbb{Z}_{k}^{n}\rightarrow\mathbb{C}$
with respect to the measure $\eta,$ that is
\[
  \left\Vert f\right\Vert _{L^{q}\left(\eta\right)}=\left(\int\left|f\right|^{q}d\eta\right)^{1/q}.
\]
When it is clear from the context which measure we are working with,
we will simply write $\left\Vert f\right\Vert _{q}.$

\subsection{\label{sub:Beckner's-inequality}A hypercontractive inequality}

Let $\nu^n$ denote the uniform measure on $\mathbb{Z}_{k}^{n}.$
%Then the inner product on $L_{\nu}\left(\mathbb{Z}_{k}^{n}\right)$
%is
%\[
%  \left\langle f,g\right\rangle _{\nu}=\int_{\mathbb{Z}_{k}^{n}}f\overline{g}d\nu=\frac{1}{k^{n}}\sum_{x\in\mathbb{Z}_{k}^{n}}f(x)\overline{g(x)}\mbox{ for }f,g\in L_{\nu}\left(\mathbb{Z}_{k}^{n}\right).
%\]
Define the ``scalar product'' on $\mathbb{Z}_{k}^{n}$ by
\[
  \left\langle x,y\right\rangle =\sum_{i=1}^{n}x_{i}y_{i},\mbox{ for }x,y\in\mathbb{Z}_{k}^{n}.
\]
Let $\varepsilon=e^{2\pi i/k}.$ For every $y\in\mathbb{Z}_{k}^{n},$
define the functions 
\[
  w_{y}\left(x\right)=\varepsilon^{\left\langle x,y\right\rangle }\mbox{ for }x\in\mathbb{Z}_{k}^{n}.
\]
It is easy to check the following lemma.
\begin{lem}
  $\left\{ w_{y}\right\} _{y\in\mathbb{Z}_{k}^{n}}$ form an orthonormal
  basis in $L_{\nu^{n}}\left(\mathbb{Z}_{k}^{n}\right).$
\end{lem}
Let us denote the number of non-zero coordinates
of $\xi\in\mathbb{Z}_{k}^{n}$ by $\left[\xi\right].$ We
will use the following hypercontractive inequality:
\begin{lem}
  (Lemma 1 of \cite{Devroye2008})\label{lem:specgen beckner}
  There are positive constants $C,\gamma$ such such that for any $k,n\in\mathbb{N},$
  $m\in\left\{ 0,1,\ldots,n\right\} $ and complex numbers $a_{y},$
  for $y\in\mathbb{Z}_{k}^{n},$ we have
  \begin{equation}
    \left\Vert \sum_{\left[y\right]=m}a_{y}w_{y}\right\Vert _{L^{4}\left(\nu^{n}\right)}\leq\left(Ck^{\gamma}\right)^{m}\left(\sum_{\left[y\right]=m}\left|a_{y}\right|^{2}\right)^{1/2}. \label{eq:}
  \end{equation}
\end{lem}
\begin{rem}
  The proof (in \cite{Devroye2008}) of Lemma \ref{lem:specgen beckner}
  is based on Claim 3.1 of \cite{N.Alon2004}. Claim 3.1 of \cite{N.Alon2004}
  is a generalization of the so called Bonami-Beckner inequality (see
  Lemma 1 of \cite{Beckner1975}). That inequality played an important
  role in \cite{Talagrand1994} in the original proof of Theorem \ref{thm:thm 1.5 of Talagrand}.
\end{rem}

\subsection{\label{sub:basis}Finding a suitable basis}

We assume that $\mu\left(\left\{ j\right\} \right)>0$ for all
$j\in\mathbb{Z}_{k}.$ Let $L_{\mu}\left(\mathbb{Z}_{k}\right)$
be the Hilbert space of functions from $\mathbb{Z}_{k}$ to $\mathbb{C},$
with the inner product
\[
  \left\langle a,b\right\rangle _{\mu}=\sum_{j\in\mathbb{Z}_{k}}a\left(j\right)\overline{b\left(j\right)}\mu\left(\left\{ j\right\} \right)\,\mbox{for }a,b\in L_{\mu}\left(\mathbb{Z}_{k}\right).
\]
Let $c_{0}\in L_{\mu}\left(\mathbb{Z}_{k}\right)$ be the constant $1$ function. By Gram-Schmidt orthogonalization, there
exist functions $c_{l}\in L_{\mu}\left(\mathbb{Z}_{k}\right)$ 
for $l\in\mathbb{Z}_{k}\setminus\left\{ 0\right\} ,$ such that $c_{j},\, j\in\mathbb{Z}_{k}$
form an orthonormal basis in $L_{\mu}\left(\mathbb{Z}_{k}\right).$

Using the functions $c_{j},$ $j\in\mathbb{Z}_{k}$ we define an orthonormal
basis in $L_{\mu^n}\left(\mathbb{Z}_{k}^{n}\right)$ analogous
to the basis $w_{y},$ $y\in\mathbb{Z}_{k}^{n}.$ It is easy to check the following lemma.
\begin{lem}
  \label{lem:u onb}The functions $u_{y},$ for $y\in\mathbb{Z}_{k}^{n},$ defined by
  \begin{equation}
    u_{y}(x)=\prod_{i=1}^{n}c_{y_{i}}(x_{i})\mbox{ for }x\in\mathbb{Z}_{k}^{n}, \label{eq:def u}
  \end{equation}
  form an orthonormal basis in $L_{\mu}\left(\mathbb{Z}_{k}^{n}\right).$
\end{lem}
%\begin{proof}
%  Since $\mu=\mu^{n}=\mu^{1}\otimes\mu^{1}\otimes\ldots\otimes\mu^{1},$
%  we have
%  \begin{eqnarray*}
%    \left\langle u_{y},u_{z}\right\rangle _{\mu} 
%    & = \int\prod_{i=1}^{n}c_{y_{i}}(x_{i})\overline{c_{z_{i}}(x_{i})}\mu(dx) \\
%    & = \prod_{i=1}^{n}\left\langle c_{y_{i}},c_{z_{i}}\right\rangle _{p} \\
%    & =
%    \begin{cases}
%      1 & \mbox{if }y=z \\
%      0 & \mbox{if }y\neq z.
%    \end{cases}
%  \end{eqnarray*}
%\end{proof}

\subsection{\label{sub:ext lemma 2.1} Extension of Lemma \ref{lem:specgen beckner}}

The key ingredient in the proof of Theorem \ref{thm:Thm1.5gen} is the following generalization of Lemma \ref{lem:specgen beckner}.
It can also be seen as an extension of Lemma 2.1 of \cite{Talagrand1994}. One could also use Theorem 2.2 of \cite{Wolff2007}, however the proof of that theorem 
is more complicated.
\begin{lem}
  \label{lem:Gen Lemma 2.1} With the constants of Lemma \ref{lem:specgen beckner}, we have for every $k,n\in\mathbb{N},$
  $m\in\left\{ 0,1,\ldots,n\right\} $ and complex numbers $a_{y},$ $y\in\mathbb{Z}_{k}^{n},$ 
  \begin{equation}
    \left\Vert \sum_{\left[y\right]=m}a_{y}u_{y}\right\Vert _{L^{4}\left(\mu^{n}\right)}\leq\left(C\theta k^{\gamma}\right)^{m}\left(\sum_{\left[y\right]=m}\left|a_{y}\right|^{2}\right)^{1/2}  \label{eq:Gen Lemma 2.1}
  \end{equation}
  holds, where $\theta=k\max_{i,j}\left|c_{i}\left(j\right)\right|.$
\end{lem}
%\begin{rem}
%  The choice for $L^4$ norm is not crucial in the right hand side of \eqref{eq:Gen Lemma 2.1} -- the same statement, with different constants, is true for all $L^q$ norms when $q\geq2.$
%  We only prove this special case, since this more general version does not improve on our main result significantly.
%  %Moreover, this way we can omit the proof of Lemma \ref{lem:specgen beckner} for general $L^q$ norms, which is a rather standard computation.
%\end{rem}

\begin{proof}
  The proof generalizes the symmetrization technique of the proof of Lemma 2.1 of \cite{Talagrand1994}. Recall the definitions
  of $\varepsilon$ and $w_{y}$ for $y\in\mathbb{Z}_{k}^{n}$ of Section \ref{sub:Beckner's-inequality}. Let $n,k,m$ and the numbers $a_{y}$
  $y\in\mathbb{Z}_{k}^{n}$ as in the statement of Lemma \ref{lem:specgen beckner}.

\textbf{Step 1} Define the product space $G=\left(\mathbb{Z}_{k}^{n}\right)^{k}$ with the probability
  measure $\mu^{n}_{k}=\bigotimes_{i=1}^{k}\mu.$ %Note that the measures $\mu_{k}$ and $\mu^{k}$ are different, the first is a measure on $G,$ while the second is a measure on $\mathbb{Z}_{k}^{k}.$
   For $y,z\in\mathbb{Z}_{k}^{n}$ define the functions $g_{y},g_{y,z}$
  on $G$ as follows. For $X=\left(X^{0},\ldots,X^{k-1}\right)\in\left(\mathbb{Z}_{k}^{n}\right)^{k}$
  and $z\in\mathbb{Z}_{k}^{n},$ let
  \begin{align}
    g_{y}\left(X\right)
    & = \prod_{1\leq i\leq n,\, y_{i}\neq0}\sum_{l=0}^{k-1}c_{y_{i}}(X_{i}^{l})\varepsilon^{ly_{i}}, \label{eq:def g_y} \\
    g_{y,z}\left(X\right)
    & = \prod_{1\leq i\leq n,\, y_{i}\neq0}\varepsilon^{z_{i}y_{i}}\sum_{l=0}^{k-1}c_{y_{i}}(X_{i}^{l})\varepsilon^{ly_{i}}=g_{y}\left(X\right)w_{y}(z). \label{eq:def g_y,z}
  \end{align}
  Recall that $\nu$ is the uniform measure on $\mathbb{Z}_{k}^{n},$
  and define the set $H=G\times\mathbb{Z}_{k}^{n}$ and the product
  measure $\kappa=\mu_{k}\otimes\nu$ on $H.$ We also define, for $y\in\mathbb{Z}_{k}^{n}$
  the functions $h_{y}$ on $H$ by $h_{y}\left(X,z\right)=g_{y,z}\left(X\right)=g_{y}\left(X\right)w_{y}(z).$

\textbf{Step 2} For $X$ as before and for $z\in\mathbb{Z}_{k}^{n}$ define $X_{z}$ as
  \[
    \left(X_{z}\right)_{i}^{l}=X_{i}^{l+z_{i}\mbox{ mod }k}.
  \]
  Then
  \begin{align*}
    g_{y,z}\left(X_{z}\right)
    & = \prod_{1\leq i\leq n,\, y_{i}\neq0}\sum_{l=0}^{k-1}c_{y_{i}}(X_{i}^{l+z_{i}\mbox{ mod }k})\varepsilon^{(l+z_{i})y_{i}} \\
    & = \prod_{1\leq i\leq n,\, y_{i}\neq0}\sum_{l=0}^{k-1}c_{y_{i}}(X_{i}^{l})\varepsilon^{ly_{i}} = g_{y}\left(X\right).    
  \end{align*}

  Hence for each fixed $z\in\mathbb{Z}_{k}^{n},$ we have 
  \begin{eqnarray}
    \left\Vert \sum_{\left[y\right]=m}a_{y}g_{y}\right\Vert _{L^{4}\left(\mu^{n}_{k}\right)} & = & \left\Vert \sum_{\left[y\right]=m}a_{y}g_{y,z}\right\Vert _{L^{4}\left(\mu^{n}_{k}\right)}.\label{eq:pf gen lemma 2.1 - 1}
  \end{eqnarray}
  Integrating over the variable $z$ with respect to $\nu^{n},$ Fubini's theorem gives that
  \begin{equation}
    \left\Vert \sum_{\left[y\right]=m}a_{y}g_{y}\right\Vert _{L^{4}\left(\mu^{n}_{k}\right)}=\left\Vert \sum_{\left[y\right]=m}a_{y}h_{y}\right\Vert _{L^{4}\left(\kappa\right)}.\label{eq:gen lemma 2.1 - 2}
  \end{equation}

\textbf{Step 3} For fixed $X,$ use Lemma \ref{lem:specgen beckner} for the numbers $a_{y}g_{y}\left(X\right),$ and get
  \begin{equation}
    \int\left|\sum_{\left[y\right]=m}a_{y}g_{y}\left(X\right)w_{y}\left(z\right)\right|^{4}d\nu^{n}(z)\leq\left(Ck^{\gamma}\right)^{4m}\left(\sum_{\left[y\right]=m}\left|a_{y}g_{y}\left(X\right)\right|^{2}\right)^{2}.\label{eq:gen lemma 2.1 - 3}
  \end{equation}
  Since $\theta=k\max_{i,j}\left|c_{i}\left(j\right)\right|$, we have that $\left|g_{y}\left(X\right)\right|\leq\theta^{m}$, which together
  with \eqref{eq:gen lemma 2.1 - 3} gives 
  \[
    \int\left|\sum_{\left[y\right]=m}a_{y}g_{y}\left(X\right)w_{y}\left(z\right)\right|^{4}d\nu^{n}(z)\leq\left(C\theta k^{\gamma}\right)^{4m}\left(\sum_{\left[y\right]=m}\left|a_{y}\right|^{2}\right)^{2}.
  \]
  Integrating with respect to $d\mu_{k}(X)$ and taking the $4$th root gives
  \begin{equation}
    \left\Vert \sum_{\left[y\right]=m}a_{y}h_{y}\right\Vert _{L^{4}\left(\kappa\right)}\leq\left(C\theta k^{\gamma}\right)^{m}\left(\sum_{\left[y\right]=m}\left|a_{y}\right|^{2}\right)^{1/2}.\label{eq:gen lemma 2.1 - 3,5}
  \end{equation}
  By \eqref{eq:gen lemma 2.1 - 3,5} and \eqref{eq:gen lemma 2.1 - 2} we only have to show that
  \begin{equation}
    \left\Vert \sum_{\left[y\right]=m}a_{y}u_{y}\right\Vert _{L^{4}\left(\mu^{n}\right)}\leq\left\Vert \sum_{\left[y\right]=m}a_{y}g_{y}\right\Vert _{L^{4}\left(\mu^{n}_{k}\right)}. \label{eq:laststepGenlemma2.1}
  \end{equation}

\textbf{Step 4 } Now we prove an alternative form of the function $g_{y}.$ Recall the definition
  \eqref{eq:def g_y} of $g_{y}.$ Expand the product, and get
  \begin{align}
    g_{y}\left(X\right)
    & = \prod_{1\leq i\leq n,\, y_{i}\neq0}\sum_{l=0}^{k-1}c_{y_{i}}(X_{i}^{l})\varepsilon^{ly_{i}}\nonumber \\
    & = \sum_{\alpha:\left(*\right)}\prod_{1\leq i\leq n,\, y_{i}\neq0}c_{y_{i}}(X_{i}^{\alpha\left(i\right)})\varepsilon^{\alpha\left(i\right)y_{i}}, \label{eq:gen lemma 2.1 - 4}   
  \end{align}
  where $\left(*\right)$ denotes the sum over all functions $\alpha:\left\{ i\,|\, y_{i}\neq0\right\} \rightarrow\mathbb{Z}_{k}.$

  We will use the following trivial observation:

  \textbf{Observation:} $c_{y_{i}}(X_{i}^{l})\varepsilon^{ly_{i}}=1$ whenever $y_{i}=0.$

  With the Observation we rewrite \eqref{eq:gen lemma 2.1 - 4} as follows.

  \begin{align}
    g_{y}\left(X\right)
    & = \sum_{\alpha\in\mathcal{A}_{y}}\prod_{i=1}^{n}c_{y_{i}}(X_{i}^{\alpha\left(i\right)})\varepsilon^{\alpha\left(i\right)y_{i}}\nonumber \\
    & = \sum_{\alpha\in\mathcal{A}_{y}}\prod_{t\in\mathbb{Z}_{k}}\prod_{1\leq i\leq n,\,\alpha\left(i\right)=t}c_{y_{i}}(X_{i}^{t})\varepsilon^{ty_{i}}, \label{eq:gen lemma 2.1 - 5}
  \end{align}
  where $\mathcal{A}_{y}$ is the set of functions $\alpha:\left\{ 1,2,\ldots,n\right\} \rightarrow\mathbb{Z}_{k}$
  with the property that $\alpha\left(i\right)=0$ if $y_{i}=0.$ For
  a function $\alpha\in\mathcal{A}_{y}$ we can define the vectors $v^{t}=v^{t}\left(\alpha\right)\in\mathbb{Z}_{k}^{n}$
  for $t\in\mathbb{Z}_{k}$ by 
  \[
    v_{i}^{t}=v_{i}^{t}\left(\alpha\right)=
    \begin{cases}
      y_{i} & \mbox{ if }\alpha\left(i\right)=t \\
      0 & \mbox{otherwise.}
    \end{cases}
  \]
  The map $\alpha\mapsto(v^{t}\left(\alpha\right))_{t\in\mathbb{Z}_{k}}$
  is one-to-one, furthermore the image of $\mathcal{A}_{y}$ under this map is
  \[
    \mathcal{V}_{y}=\left\{ v=\left(v^{t}\right)_{t\in\mathbb{Z}_{k}}\left|\sum_{t\in\mathbb{Z}_{k}}v^{t}=y,\mbox{ and }\forall i\, v_{i}^{t}\neq0\mbox{ for at most one }t\in\mathbb{Z}_{k}\,\right.\right\} .
  \]
  Using the properties of the map $\alpha\mapsto(v^{t}\left(\alpha\right))_{t\in\mathbb{Z}_{k}}$
  together with the Observation and the definition of $u,$ we can conclude from \eqref{eq:gen lemma 2.1 - 5} that 
  \begin{align}
    g_{y}\left(X\right)
    & = \sum_{v\in\mathcal{V}_{y}}\prod_{t\in\mathbb{Z}_{k}}\prod_{i=1}^{n}c_{v_{i}^{t}}(X_{i}^{t})\varepsilon^{tv_{i}^{t}}\nonumber \\
    & = \sum_{v\in\mathcal{V}_{y}}\prod_{t\in\mathbb{Z}_{k}}u_{v^{t}}(X^{t})\varepsilon^{t\left\langle v^{t},\boldsymbol{1}\right\rangle }\label{eq:gen lemma 2.1 - 6}
  \end{align}
  where $\boldsymbol{1}$ is vector in $\mathbb{Z}_{k}^{n}$ with all coordinates equal to $1.$

\textbf{Step 5} Now we prove
  \eqref{eq:laststepGenlemma2.1}. Jensen's inequality gives that
  \begin{align}
    \int\left|\sum_{\left[y\right]=m}a_{y}g_{y}\left(X\right)\right|^{4}d\mu^{n}_{k}\left(X\right) 
  & \geq \int\left|\int\sum_{\left[y\right]=m}a_{y}g_{y}\left(X\right)d\mu^{n}_{k-1}\left(X^{1},\ldots,X^{k-1}\right)\right|^{4}d\mu^{n}\left(X^{0}\right) \nonumber \\
  & = \int\left|\sum_{\left[y\right]=m}a_{y}\int g_{y}\left(X\right)d\mu^{n}_{k-1}\left(X^{1},\ldots,X^{k-1}\right)\right|^{4}d\mu^{n}\left(X^{0}\right).\label{eq:gen lemma 2.1 - 7}
  \end{align}
  By \eqref{eq:gen lemma 2.1 - 6}, the inner integral of the left hand side of \eqref{eq:gen lemma 2.1 - 7} is
  \begin{align}
    \int g_{y}\left(X\right)d\mu^{n}_{k-1}\left(X^{1},\ldots,X^{k-1}\right)
    & = \int\sum_{v\in\mathcal{V}_{y}}\prod_{t\in\mathbb{Z}_{k}}u_{v^{t}}(X^{t})\varepsilon^{t\left\langle v^{t},\boldsymbol{1}\right\rangle }d\mu^{n}_{k-1}\left(X^{1},\ldots,X^{k-1}\right) \\
    & = \sum_{v\in\mathcal{V}_{y}}\left(\prod_{t\in\mathbb{Z}_{k}}\varepsilon^{t\left\langle v^{t},\boldsymbol{1}\right\rangle }\right)u_{v^{0}}(X^{0})\prod_{l=1}^{k-1}\int u_{v^{l}}(X^{l})d\mu^{n}\left(X^{l}\right).\label{eq:gen lemma 2.1 - 9}
  \end{align}
  Since $u_{0}$ is the constant $1$ function on $\mathbb{Z}_{k}^{n},$
  and by Lemma \ref{lem:u onb} $\left(u_{w},\, w\in\mathbb{Z}_{k}^{n}\right)$
  is an orthonormal basis of $L_{\mu}\left(\mathbb{Z}_{k}^{n}\right),$ we have
  \begin{align*}
    \int u_{w}d\mu^{n}
    & =\int u_{w}u_{0}d\mu^{n} = 
    \begin{cases}
      1 & \mbox{if }w=0 \\
      0 & \mbox{otherwise.}   
    \end{cases}
  \end{align*}
  By this and the definition of $\mathcal{V}_{y}$ we conclude from \eqref{eq:gen lemma 2.1 - 9} that
  \begin{align}
    \int g_{y}\left(X\right)d\mu^{n}_{k-1}\left(X^{1},\ldots,X^{k-1}\right)
    & =\sum_{v\in\mathcal{V}_{y},\, v^{1}=\ldots=v^{k-1}=0}\left(\prod_{t\in\mathbb{Z}_{k}}\varepsilon^{t\left\langle v^{t},\boldsymbol{1}\right\rangle }\right)u_{v^{0}}(X^{0})=u_{y}(X^{0}).\label{eq:gen lemma 2.1 - 8}
  \end{align}
  \eqref{eq:gen lemma 2.1 - 8} together with \eqref{eq:gen lemma 2.1 - 7} gives that
  \[
    \int\left|\sum_{\left[y\right]=m}a_{y}g_{y}\left(X\right)\right|^{4}d\mu^{n}_{k}\left(X\right)\geq\int\left|\sum_{\left[y\right]=m}a_{y}u_{y}(X^{0})\right|^{4}d\mu^{n}\left(X^{0}\right),
  \]
  from which by taking the $4$th root, we get \eqref{eq:laststepGenlemma2.1}.
  This completes the proof of Lemma \eqref{lem:Gen Lemma 2.1}.
\end{proof}
From Lemma \eqref{lem:Gen Lemma 2.1} and duality, we conclude the following lemma.
\begin{lem}
  \label{lem:genProp2.2} With the constants of Lemma \ref{lem:specgen beckner}, for any function $g\in L_{\mu}\left(\mathbb{Z}_{k}^{n}\right)$
  we have
  \[
    \sum_{\left[y\right]=l}\left|\hat{g}\left(y\right)\right|^{2}\leq\left(C\theta k^{\gamma}\right)^{2l}\left\Vert g\right\Vert _{L^{4/3}\left(\mu\right)}^{2}.
  \]
\end{lem}

\subsection{Completion of the proof of Theorem \ref{thm:Thm1.5gen}}

Notice that
\begin{align*}
  \int_{\Omega}u_{y}\left(x_{1},\ldots,x_{i-1},\xi,x_{i+1},\ldots,x_{n}\right)\mu\left(d\xi\right) 
  & = \sum_{j\in\mathbb{Z}_{k}}c_{y_{i}}(j)\mu\left(\left\{j\right\}\right)\prod_{1\leq l\leq n\, l\neq i}c_{y_{l}}\left(x_{l}\right) \\
  & = \left\langle c_{y_{i}},c_{0}\right\rangle _{\mu}\prod_{1\leq l\leq n\, l\neq i}c_{y_{l}}\left(x_{l}\right) \\
  & = 
  \begin{cases}
    u_{y}\left(x\right) & \mbox{if }y_{i}=0 \\
    0 & \mbox{if }y_{i}\neq0.
  \end{cases}
\end{align*}
Hence
\begin{align*}
  \int_{\Omega}f\left(x_{1},\ldots,x_{i-1},\xi,x_{i+1},\ldots,x_{n}\right)\mu\left(d\xi\right) 
  & = \sum_{y\in\mathbb{Z}_{k}^{n},\, y_{i}=0}\hat{f}(y)u_{y}
\end{align*}

where $f=\sum_{y}\hat{f}(y)u_{y},$ i.e $\hat{f}(y)=\left\langle f,u_{y}\right\rangle _{\mu}.$

By the definition of $\Delta_{i}f$ we have
\begin{equation}
  \Delta_{i}f=\sum_{y\in\mathbb{Z}_{k}^{n},\, y_{i}\neq0}\hat{f}(y)u_{y}.\label{eq:deltafFourier}
\end{equation}

Recall that $\left[y\right]$ was the number of
non-zero coordinates of a vector $y\in\mathbb{Z}_{k}.$ Define $M(g)$
by
\begin{equation*}
  M\left(g\right)^{2}=\sum_{y\in\mathbb{Z}_{k}^{n},\, y\neq0}\frac{\hat{g}\left(y\right)^{2}}{\left[y\right]}\mbox{ for }g\in L_{\mu}\left(\mathbb{Z}_{k}^{n}\right).
\end{equation*}

Take a function $f\in L_{\mu}\left(\mathbb{Z}_{k}^{n}\right)$
with $\int fd\mu=0$ (which is equivalent to $\hat{f}(0)=0$). Then
Parseval's formula and \eqref{eq:deltafFourier} gives that
\begin{equation}
  \left\Vert f\right\Vert _{L^{2}\left(\mu^{n}\right)}^{2}=\sum_{y\neq0}\hat{f}(y)^{2}=\sum_{i=1}^{n}M(\Delta_{i}f)^{2}. \label{eq:Thm1.5 pf - 1}
\end{equation}

Since $1=\sum_{j=0}^{k-1}\left|c_{i}\left(j\right)\right|^{2}p_{j},$
we can conclude that $\theta\leq k/\min_{j}\sqrt{p_{j}}.$ Hence Theorem
\ref{thm:Thm1.5gen} follows from the Proposition \ref{pro:M(g)^2 upper bound} bellow
and \eqref{eq:Thm1.5 pf - 1}.
\begin{prop} 
  \label{pro:M(g)^2 upper bound} There is a positive
  constant $K,$ such that if $\int gd\mu=0,$ we have 
  \[
    M(g)^{2}\leq K\log\left(C\theta k^{\gamma}\right)\frac{\left\Vert g\right\Vert _{2}^{2}}{\log\left(e\left\Vert g\right\Vert _{2}/\left\Vert g\right\Vert _{1}\right)},
  \]
  where $\theta=k\max_{i=1,\ldots,n\, j\in\mathbb{Z}_{k}}\left|c_{i}\left(j\right)\right|,$
  and the constants $C,\gamma$ are the same as in Lemma \ref{lem:specgen beckner}.
\end{prop}

\begin{proof}
  The proof of Proposition \eqref{pro:M(g)^2 upper bound}
  is the same as the proof of Proposition 2.3 in \cite{Talagrand1994}
  with the following modifications. Take $q=4$ instead of $q=3$, and
  use Lemma \ref{lem:genProp2.2} instead of Proposition 2.2 of \cite{Talagrand1994}.
  The only difference will be in the constants. First we get the term
  $2\log\left(C\theta k^{\gamma}\right)$ in stead of $\log\left(2\theta^{2}\right)$.
  Furthermore we have to replace the estimate
  \[
    \frac{\left\Vert g\right\Vert _{2}}{\left\Vert g\right\Vert _{1}}\leq\left(\frac{\left\Vert g\right\Vert _{2}}{\left\Vert g\right\Vert _{3/2}}\right)^{3}
  \]
  by
  \[
    \frac{\left\Vert g\right\Vert _{2}}{\left\Vert g\right\Vert _{1}}\leq\left(\frac{\left\Vert g\right\Vert _{2}}{\left\Vert g\right\Vert _{4/3}}\right)^{2},
  \]
  which is a consequence of the Cauchy-Schwartz inequality. This substitution only affects the constant $K.$

  This completes the proof of Proposition \eqref{pro:M(g)^2 upper bound}
  and the proof of Theorem \ref{thm:Thm1.5gen}.
\end{proof}

\subsubsection*{Acknowledgment.}

%I thank Hamed Hatami for discussions concerning my counterexample to Theorem 3.3 in his paper \cite{Hatami2009}. Further, 
I thank Rob van den Berg, for introducing me to the subject, and for the comments on drafts of this paper.

\bibliographystyle{abbrv}
\bibliography{/ufs/demeter/LyX/References/myreflist}

\end{document}